\documentclass[11pt,letterpaper,openany]{article}
\usepackage{amsmath}\usepackage{amssymb}\usepackage{latexsym}\usepackage{amsfonts}
\usepackage[latin1]{inputenc}\usepackage[english]{babel}\usepackage{multicol}
\usepackage{oldgerm}\usepackage{enumerate}

\newtheorem{theorem}{Theorem}  
\newtheorem{corollary}{Corollary} \newtheorem{Definition}{Definition} \newtheorem{example}{Example}
 \newtheorem{proposition}{Proposition} 

\newenvironment{proof}[1][Proof]{\noindent\textbf{#1.} }{\ \rule{0.5em}{0.5em}}

\def\F{\mathbb{F}}
\def\Z{\mathbb{Z}}
\def\mA{\mathcal{A}}\def\mB{\mathcal{B}}

\begin{document}

\title{New Constructions of Sonar Sequences}

\author{Diego F. Ruiz\footnote{Department of Mathematics, Universidad del Cauca, Colombia.} , Carlos A. Trujillo$^*$, and Yadira Caicedo\footnote{Student of Doctorates in Mathematics at Universidad del Valle, Colombia.}\\{\small dfruiz@unicauca.edu.co, trujillo@unicauca.edu.co, yadira0427@gmail.com}}
\date{}
\maketitle

\begin{abstract}
A set $\mA$ is a Sidon set in an additive group $G$ if every element of $G$ can be written at most one way as sum of two elements of $\mA$. A particular case of two--dimensional Sidon sets are the sonar sequences, which are two--dimensional synchronization patterns. The main known constructions of sonar sequences are reminiscent of Costas arrays constructions (Welch and Golomb). Other constructions are Quadratic and Shift. In this work we present new constructions of sonar sequences, independent of the named above, using one--dimensional Sidon sets.
\medskip



\vspace{0.5cm}

\end{abstract}


\section{Introduction}
Sidon sets is a problem in additive number theory dating from 1930's with a big impact in engineering, specially in communication theory. These sets receive its name due to Simon Sidon who introduced them in order to solve an harmonic analysis problem. Sidon wondered about the existence of sets of positive integers in which all sums of two elements of the set are different, which is equivalent to establish that all differences of two elements (distincts) of that set are also different \cite{GolombRulersMathMag}. Let $n$ be a positive integer. Formally, $\mA=\{a_1,\ldots,a_n\}\subset \mathbb{Z}$ is a Sidon set if, for all $x\in \mathbb{Z}$ and all $a_i,a_j\in \mA$ with $i\leq j$, the amount of solutions of the equation $x=a_i+a_j$ is 0 or 1. Though Sidon sets were initially introduced in $\mathbb{Z}$, it is possible to construct them in other groups and even any dimension \cite{ProductOfGroups}.

Applications of Sidon sets include radio communications, x-ray crystallography, self--orthogonal codes, linear arrays in the formation of optimal linear telescope arrays, and pulse phase modulation (PPM) as a form of communications \cite{Rankin93optimalgolomb}.

Two--dimensional Sidon sets are used in applications to sonar and radar systems, which deal with the following problem: ``How can an observer determine the distance and velocity of an object that is moving?''. The solution to this problem use the Doppler effect, which states that when a signal bounces off a moving target its frequency changes in direct proportion to the velocity of the object relative to the observer. That is, if the observer sends out a signal towards a moving target, the change between the frequency of the outgoing and that of the returning signal will allow us to determine the velocity of the target, and the time it took to make the round trip will allow us to determine the distance \cite{MorenoExhaustive}.

Sidon sets in sonar and radar systems are represented by a matrix $A$ with entries 0 or 1, and the signal consists of available frequencies  $f_1,\ldots,f_m$ to be transmitted at each one of the consecutive time intervals $t_1,\ldots,t_n$, with $m$ a positive integer. According to this consideration, the entry $a_{ij}$ of $A$ is 1 if frequency $f_i$ is transmitted in time interval $t_j$, otherwise is 0 \cite{CostasArraysConstructions, SurveyCostas}. Note that in these systems there is just one 1 per column. Other applications of two--dimensional Sidon sets appear in optimal communication systems, cryptography, key distribution in cellular networks, military satellite communications, and other areas \cite{SurveyCostas}.

Respecting to sonar sequences, they were introduced as examples of two--dimensional synchronization patterns with minimum ambiguity \cite{TwoDimensional, Sonar}. A sequence $[f_1,\ldots,f_n]$ is a sonar sequence of length $n$ over the set of integers $\{1,2,\ldots,m\}$ (or an $m\times n$ sonar sequence) if for all integers $h,i,$ and $j$, with $1\leq h\leq n-1$ and $1\leq i,j\leq n-h$,
\begin{equation}
    f_{i+h}-f_{i}=f_{j+h}-f_{j}\quad\Rightarrow i=j.
    \label{Sonar}
\end{equation}
If \eqref{Sonar} is replaced by
\begin{equation}
    f_{i+h}-f_{i}\equiv f_{j+h}-f_{j}\pmod m\quad\Rightarrow i=j,
    \label{SonarModular}
\end{equation}
the sequence is a sonar sequence of length $n$ and module $m$ (or an $m\times n$ modular sonar sequence).

In \cite{Sonar} are presented algorithms which allow us to construct modular sonar sequences, some of them reminiscent of Costas array constructions. Furthermore, is presented a way to transform a given sequence in another in order to be close to the main problem for sonar sequence which states:
\[
    \text{``For a fixed $m$, find the largest $n$ for which there exists an $m\times n$ sonar sequence''}.
\]
Let $p$ be a prime and $r$ a positive integer. If $q=p^r$, $\mathbb{F}_q$ denotes the finite field with $q$ elements. Following we present the sonar sequences constructions given in \cite{Sonar}.

\begin{theorem}[Quadratic construction]
Let $p$ be a odd prime, and $a,b,c$ integers such that $a\not\equiv 0\pmod p$. The sequence $[f_1,\ldots,f_{p+1}]$ defined by $f_i:=ai^2+bi+c\pmod p$ is a $p\times (p+1)$ modular sonar sequence.
\end{theorem}

\begin{theorem}[Shift construction]
Let $q=p^r$ be a prime power, $\alpha$ a primitive element of $\mathbb{F}_{q^2}$ and $\beta$ a primitive element of $\mathbb{F}_{q}$. The sequence $[f_1,\ldots,f_q]$ defined by $f_i:=\log_{\beta}(\alpha^{iq}+\alpha^i)$ is a $(q-1)\times q$ modular sonar sequence.
\end{theorem}

\begin{theorem}[Extended exponential Welch construction]
Let $\alpha$ be a primitive element module $p$ and $s$ an integer. The sequence $[f_0,f_1,\ldots,f_{p-1}]$ defined by $f_i:=\alpha^{i+s}$ is a $p\times p$ modular sonar sequence. If $s=0$, the sequence $[f_1,\ldots,f_{p-1}]$ defined as above is a $p\times (p-1)$ modular sonar sequence.
\end{theorem}

\begin{theorem}[Logarithmic Welch construction]
Let $\alpha$ be a primitive element module $p$. The sequence $[f_1,\ldots,f_{p-1}]$ defined by $f_i:=\log_{\alpha}i$ is a $(p-1)\times(p-1)$ modular sonar sequence.
\end{theorem}

\begin{theorem}[Golomb construction]
Let $q$ be a prime power greater than 2, and let $\alpha,\beta$ be primitive elements $\mathbb{F}_q$. The sequence $[f_1,\ldots,f_{q-2}]$ defined by $f_i:=j$ if and only if $\alpha^i+\beta^j=1$ is a $(q-1)\times (q-2)$ modular sonar sequence. If $\alpha=\beta$ this construction is known as Lempel construction.
\end{theorem}

The main object of this paper is to present new constructions of sonar sequences independent of those given above. Our new construction (Theorem~\ref{new sonar general}) uses specials one--dimensional Sidon sets that are useful to establish it. We use this construction and apply some properties we identify in the Sidon sets--constructions due to Bose and due to Ruzsa to state three new constructions of sonar sequences.

The remain of this work is as follows: In Section 2 we present constructions of Sidon sets of Bose and Ruzsa (Theorems~\ref{ConsBose} and \ref{ConsRuzsa}) and some properties that emerge from those constructions (Propositions~\ref{proposition Bose} and \ref{proposition Ruzsa}) and that are used to present the new construction of sonar sequence (Theorem~\ref{new sonar general}) and its immediate results (Corollaries~\ref{Bose1} and \ref{CRuzsa1}) in Section 3. Finally, in Section 4 we present some related problems with this work.

In the following sections, if $a$ and $m$ are integers such that $a<m$, then $[a,m]$ denotes the finite set $\{a,a+1,\ldots,m\}$. Furthermore, if $G$ is a cyclic group generated by $g\in G$, we denote $G=\langle g\rangle$.

\section{Construction of Sidon sets of Bose and Ruzsa}
Although initially Sidon sets were introduced in the additive group $\Z$, they can be defined in a finite abelian group $(G,*)$ as follows.

\begin{Definition}
    Let $(G,*)$ be a finite abelian group. $\mA=\{a_1,\ldots,a_k\}\subseteq G$ is a Sidon set in $G$ if for all $a_i,a_j\in \mA$, with $1\leq i\leq j\leq k$, the amount of solutions of the equation $x=a_i*a_j$ is at most 1.
\end{Definition}

Note that this definition is equivalent to prove that the amount of solutions of the equation $x=a*b^{-1}$ is at most 1, for any $a,b\in\mA$ and $a\neq b$.

To prove that $\mA\subseteq G$ is a Sidon set we will use one of the following properties
\begin{enumerate}
    \item[(S1)] If $a*b=c*d$ then $\{a,b\}=\{c,d\}$, for any $a,b,c,d\in\mA$.
    \item[(S2)] Let $G$ and $G^\prime$ be two abelian groups. If $\varphi: G \rightarrow G'$ is an isomorphism and $\mA$ is a Sidon set in $G$, then $\varphi(\mA)$ is a Sidon set on $G'$.
\end{enumerate}

The main problem in Sidon sets consists in to establish the largest cardinality of a set satisfying Definition 1. That is, to study the asymptotically behavior of the function
\[
    f_2(G):=\max\{|\mA|:\,\mA\subseteq G \text{ is a Sidon set}\}.
\]
The value of the previous function for any group in general is not yet known, however there exist the following estimated for some upper and lower bounds.

Using counting techniques we can prove that
\begin{equation}
    f_2(G)\leq \left\lfloor\frac{1+\sqrt{4|G|-3}}{2}\right\rfloor.
    \label{trivial bound f2}
\end{equation}
Now, using three Sidon--set constructions (Ruzsa \cite{RuzaSolving}, Bose \cite{Bose}, and Singer \cite{Singer}) and \eqref{trivial bound f2} we can state $f_2(\Z_{p^2-p})=p-1$, $f_2(\Z_{q^2-1})=q$, and $f_2(\Z_{q^2+q+1})=q-1$.

In order to present new constructions of sonar sequences in Section III, we introduce the Sidon--set constructions of Bose and Ruzsa. 

\subsection{Bose's construction}

Let $n\in\mathbb{N}$ and $p$ be a prime number. Consider the finite field $\F_{q}$ with $q=p^{n}$ elements. Let $\alpha$ be an algebraic element of degree $2$ over $\F_{q}$; the degree of the extension field $\F_{q}\left(\alpha\right)$ over $\F_{q}$ is 2, furthermore $\F_{q}\left(\alpha\right)=\F_{q^{2}}$. Because $\F_{q^2}$ is a finite field, we have that $\F_{q^2}^*=\F_{q^2}\setminus\{0\}$ is a cyclic group. Let $\theta$ be a primitive element of $\F_{q^2}$. Define the set
\begin{equation}
    \mB\left(\alpha\right)=\alpha+\F_{q}:=\left\{ \alpha+a:\: a\in\F_{q}\right\} \subseteq\F_{q^{2}}.
    \label{B_de_alfa}
\end{equation}
Note that $|\mB(\alpha)|=q$.

\begin{theorem}\label{B de alfa es conjunto B2 multiplicativo}
    $\mB\left(\alpha\right)$ is a Sidon set in the multiplicative group $\F_{q^{2}}^{*}$.
\end{theorem}
\begin{proof}
Suppose there exist $a,b,c,d\in\F_{q}$ such that
\begin{equation}
    (\alpha+a)(\alpha+b)=(\alpha+c)(\alpha+d).
    \label{First Lemma1}
\end{equation}
From \eqref{First Lemma1} we have that $\alpha$ satisfies the polynomial in $\F_{q}\left[x\right]$ of degree 1
\[
    P(x)=\left[a+b-\left(c+d\right)\right]x+ab-cd
\]
Because $\alpha$ is an algebraic element of degree 2 over the field $\F_{q}$, the minimal polynomial also have degree 2, so $P(x)$ must be the zero polynomial. Thus we have $a+b=c+d$ and $ab=cd$, that is possible only if $\left\{a,b\right\} =\left\{ c,d\right\}$, implying $\left\{\alpha+a,\alpha+b\right\} =\left\{ \alpha+c,\alpha+d\right\}$, that is, $\mB(\alpha)$ is a Sidon set in the multiplicative group $\F_{q^{2}}^{*}$.
\end{proof}
\vspace{0.2cm}

Now, using (S2) and Theorem \ref{B de alfa es conjunto B2 multiplicativo}, we can state Bose's construction.

\begin{theorem}\label{ConsBose}
    Let $\alpha$ be an algebraic element of degree 2 over $\F_{q}$ and $\theta$ be a primitive element of $\F_{q^{2}}$. The set
    \begin{equation}\label{log B de alfa}
    \mB(q,\theta,\alpha)=\log_{\theta}\left(\mB(\alpha)\right):=\{\log_{\theta}(\alpha+a):\, a\in\F_{q}\}
    \end{equation}
    is a Sidon set with $q$ elements in the abelian additive group $\Z_{q^{2}-1}$.
\end{theorem}
\begin{proof}
Let $\theta$ be a primitive element of $\F_{q^{2}}$. Using the discrete logarithm to the base $\theta$, the multiplicative group $\F_{q^{2}}^{*}$ is isomorphic to the additive group $\Z_{q^{2}-1}$. Because $\mB\left(\alpha\right)$ defined in Theorem \ref{B de alfa es conjunto B2 multiplicativo} is a Sidon set in $\F_{q^{2}}^*$ then $\mB(q,\theta,\alpha)=\log_\theta\mB(\alpha)$ is a Sidon set in $\Z_{q^{2}-1}$.
\end{proof}

The following Proposition presents some properties in a Sidon set obtained from Bose's construction. As above, let $\theta$ be a primitive element of $\F_{q^2}$, and $\alpha\in \F_{q^2}$ an algebraic element of degree $2$ over $\F_q$. Also, we denote $\mB(q,\theta)=\mB(q,\theta,\alpha)$.

\begin{proposition}\label{proposition Bose}
     Let $\mB(q,\theta)$ be the Sidon set given in $\eqref{log B de alfa}$. Then
     \begin{enumerate}
        \item[(B1)] If $a\in \mB(q,\theta)$ then $a\not\equiv0\bmod(q+1)$.\smallskip
        \item[(B2)] If $a,b\in\mB(q,\theta)$ with $a\neq b$, then $a\not\equiv b\bmod(q+1)$.\smallskip
        \item[(B3)] $\mB(q,\theta)\bmod(q+1):=\{a\bmod(q+1):\, a\in \mB(q,\theta)\}=[1,q]$.\smallskip
            \end{enumerate}
\end{proposition}
\begin{proof}
    \begin{enumerate}
        \item[(B1)] If $a\in \mB(q,\theta)$, there exists $k\in\F_{q}$ such that $a=\log_{\theta}(\alpha+k)$, implying that $\theta^{a}=\alpha+k$. Because $\alpha\not\in\F_{q}^{*}=\left<\theta^{q+1}\right>$ we have that $\theta^{a}\not\in\F_{q}^{*}$. Therefore $q+1$ does not divide $a$ and so we have the desired result.\medskip
        \item[(B2)] We verify this property by contradiction. Suppose $a\equiv b\bmod(q+1)$, that is, there exists $t\in\Z$ such that $a-b=t(q+1)$. So $\theta^{a-b}=\theta^{t(q+1)}$ and we have $\theta^{a-b}\in\F_{q}^{*}$. On the other hand, since $a,b\in \mB(q,\theta)$, there exist $k_1,k_2\in\F_q$, $k_1\neq k_2$ such that $a=\log_{\theta}(\alpha+k_{1})$ and $b=\log_{\theta}(\alpha+k_{2})$, implying that $\theta^{a}=\alpha+k_{1}$ and $\theta^{b}=\alpha+k_{2}$. Thus $\theta^{a-b}=\frac{\alpha+k_{1}}{\alpha+k_{2}}$. Since $\theta^{a-b}\in\F_{q}^{*}$, there exists $c\in\F_{q}^{*}$ such that $\frac{\alpha+k_{1}}{\alpha+k_{2}}=c$. So, $(1-c)\alpha=ck_{2}-k_{1}$. Because $c\neq1$, $\alpha=(ck_{2}-k_{1})(1-c)^{-1}\in\F_{q}^{*}$ what is a contradiction.
            \medskip
        \item[(B3)] Since $|\mB(q,\theta)|=q$, using (B1) and (B2) we have $\mB(q,\theta)\bmod(q+1)=[1,q]$.\medskip
    \end{enumerate}
\end{proof}

\subsection{Ruzsa's construction}
Let $p$ be a prime number and let $\F_p$ the finite field with $p$ elements.

\begin{theorem}\label{TRuzsa1}
For any primitive element $\theta$ of $\F_p$, the set
\[
    \mathcal{R}:=\{(i,\theta^i):1\leq i\leq p-1\}
\]
is a Sidon set with $p-1$ elements in the abelian additive group $\Z_{p-1}\times\F_p$.
\end{theorem}

\begin{proof}
Let $(i,\theta^i),(j,\theta^j),(k,\theta^k),(\ell,\theta^\ell)$ be elements of $\mathcal{R}$ such that
\begin{equation}
    (i,\theta^i)+(j,\theta^j)=(k,\theta^k)+(\ell,\theta^\ell).
    \label{Ruzsa1}
\end{equation}
From \eqref{Ruzsa1} we have $i+j\equiv k+\ell\bmod (p-1)$ and $\theta ^i+\theta ^j=\theta ^k+\theta ^\ell$ in $\F_p$. Because $\theta$ is a primitive element of $\F_p$, it follows
\begin{align*}
    \theta ^i+\theta ^j & =\theta ^k+\theta ^\ell,\\
    \theta^{i}\theta^{j} & =\theta^{k}\theta^{\ell},
\end{align*}
what implies that $\{\theta^i,\theta^j\}=\{\theta^k,\theta^\ell\}$ and $\{i,j\}=\{k,l\}$. Therefore
\[
    \{(i,\theta^i),(j,\theta^j)\}=\{(k,\theta^k),(\ell,\theta^\ell)\},
\]
that is, $\mathcal{R}$ is a Sidon set in the abelian additive group $\Z_{p-1}\times\F_p$.
\end{proof}

Now, using (S2) and Theorem \ref{TRuzsa1}, we can state Ruzsa's construction.

\begin{theorem}\label{ConsRuzsa}
Let $\theta$ be a primitive element of $\F_{p}$. Then
\begin{equation}
    \mathcal{R}(p,\theta):=\{x\equiv ip-{\theta}^i(p-1) \bmod (p^2-p): 1\leq  i\leq p-1\}
    \label{conjunto de Ruzsa}
\end{equation}
is a Sidon set with $p-1$ elements in the additive group $\Z_{p^2-p}$.
\end{theorem}

\begin{proof}
By Theorem \ref{TRuzsa1} we know that
\[
    \mathcal{R}=\{(i,\theta^i):1\leq i\leq p-1\},
\]
is a Sidon set with $p-1$ elements in the additive group $\Z_{p-1}\times\F_p$ .

Let $\varphi:\Z_{p-1}\times\Z_p \rightarrow\Z_{p(p-1)}$ be the isomorphism defined by the Chinese remainder theorem, we have the following system of congruences
\begin{align*}
    \varphi(i,{\theta}^i)&\equiv i \bmod (p-1),\\
    \varphi(i,{\theta}^i)&\equiv {\theta}^i \bmod p,
\end{align*}
for all $(i,\theta^i)\in \mathcal{R}$.

To solve the above system, let $x:=\varphi(i,{\theta}^i)$. So
\begin{align}
    x&\equiv i\bmod (p-1) \label{equ 1},\\
    x&\equiv {\theta}^i \bmod p \label{equ 2},
\end{align}

From \eqref{equ 1}, there exists $t\in\Z$ such that $x=i+t(p-1)$. Replacing $x$ in \eqref{equ 2} we get
\begin{align*}
    i+t(p-1) & \equiv {\theta}^i \bmod p\\
    t & \equiv   i-{\theta}^i \bmod p
\end{align*}
what implies that there exists $s\in\Z$ such that $t = i-{\theta}^i+sp$. So, $x=ip-{\theta}^i(p-1)+sp(p-1)$ and since $x\in\Z_{p^2-p}$ we have $x=ip-{\theta}^i(p-1)$.

Therefore, the image of $ \mathcal{R}$ through $\varphi$ is $\mathcal{R}(p,\theta)$, and so by (S2) we have the desired result.
\end{proof}

Using the Sidon set given in Theorem \ref{ConsRuzsa}, in the following proposition we present some properties that it satisfies.
\begin{proposition}\label{proposition Ruzsa}
Let $\mathcal{R}(p,\theta)$ be the set given in \eqref{conjunto de Ruzsa}.
\begin{enumerate}
        \item[(R1)]$\mathcal{R}(p,\theta)\bmod{p}=[1,p-1]$.\smallskip
        \item[(R2)]$\mathcal{R}(p,\theta)\bmod(p-1)=[1,p-1]$.\smallskip
\end{enumerate}
\end{proposition}
\begin{proof}
\begin{enumerate}
         \item[(R1)] Note that
         \begin{align*}
         \mathcal{R}(p,\theta)\bmod(p)&=\{-\theta^{i}(p-1)\bmod p: 1\leq  i\leq p-1\}\\
                                      &=\{\theta^{i}\bmod p: 1\leq  i\leq p-1\}\\
                                      &=[1,p-1].
         \end{align*}

         \item[(R2)] As in the previous property, note that
         \begin{align*}
         \mathcal{R}(p,\theta)\bmod(p-1)&=\{ip\bmod(p-1): 1\leq  i\leq p-1\}\\
                                      &=\{i: 1\leq  i\leq p-1\}\\
                                      &=[1,p-1].
         \end{align*}
    \end{enumerate}
\end{proof}

\section{New Constructions of Sonar Sequences}
To introduce the notion of sonar sequences as a function consider the following definition.

\begin{Definition}
A function $f:[1,n]\rightarrow [1,m]$ has the distinct differences property if for all $i,j,h\in\mathbb{Z}$ such that $1\leq h\leq n-1$ and $1\leq i,j\leq n-h$,
\begin{equation}
    f(i+h)-f(i)=f(j+h)-f(j)\Rightarrow i=j,
    \label{DDDef}
\end{equation}
Furthermore, if $[1,m]$ is replaced by $\Z_m$ and \eqref{DDDef} by
\begin{equation}
    f(i+h)-f(i)\equiv f(j+h)-f(j)\pmod m\Rightarrow i=j,
    \label{DDModularDef}
\end{equation}
$f$ has the distinct differences property module $m$.
\end{Definition}


\medskip
Using the previous definition we can state the concept of sonar sequence.
\begin{Definition}
An $m\times n$ sonar sequence is a function $f:[1,n]\rightarrow [1,m]$ with the distinct differences property. Furthermore, an $m\times n$ modular sonar sequence is a function $f:[1,n]\rightarrow \Z_m$ satisfying the distinct differences property module $m$.
\end{Definition}

\medskip
In the following theorem we present the main result of this work that shows a new construction of modular sonar sequences using Sidon sets.

\begin{theorem}\label{new sonar general}
Let $m,b\in\mathbb{N}$ and $\mA=\{a_1,\ldots,a_n\}$ be a Sidon set in the additive group $\mathbb{Z}_{mb}$. If $\mA \bmod b:=\{a\bmod b:a\in \mA\}=[1,n]$, then the function $f:[1,n]\rightarrow\Z_m$ defined by $f(i)=\left\lfloor\frac{a_i}{b}\right\rfloor$, is an $m\times n$ modular sonar sequence, where $a_i$ is the unique element such that $a_i\equiv i\bmod b$ and $\left\lfloor y\right\rfloor$ denotes the smallest integer less than or equal to $y$.
\end{theorem}

\begin{proof}
Let $i,j,h\in[1,n]$ such that $1\leq h\leq n-1$ and $1\leq i,j\leq n-h$. Suppose that
\begin{equation*}
    f(i+h)-f(i)\equiv f(j+h)-f(j)\pmod m,
    \label{Cons1}
\end{equation*}
that is
\begin{equation*}
    \left\lfloor\frac{a_{i+h}}{b}\right\rfloor-\left\lfloor\frac{a_i}{b}\right\rfloor \equiv
    \left\lfloor\frac{a_{j+h}}{b}\right\rfloor-\left\lfloor\frac{a_j}{b}\right\rfloor\pmod m.
\end{equation*}
Thus, there exists $t\in\mathbb{Z}$ such that
\begin{equation}
    \left\lfloor\frac{a_{i+h}}{b}\right\rfloor-\left\lfloor\frac{a_i}{b}\right\rfloor=
    \left\lfloor\frac{a_{j+h}}{b}\right\rfloor-\left\lfloor\frac{a_j}{b}\right\rfloor +tm.
    \label{Cons1}
\end{equation}
By multiplying \eqref{Cons1} by $b$ we have
\begin{equation}
    b\left\lfloor\frac{a_{i+h}}{b}\right\rfloor-b\left\lfloor\frac{a_i}{b}\right\rfloor=
    b\left\lfloor\frac{a_{j+h}}{b}\right\rfloor-b\left\lfloor\frac{a_j}{b}\right\rfloor + tmb.
    \label{Cons2}
\end{equation}
Because for each $a_k\in\mA$ is satisfied that $a_k=\left\lfloor\frac{a_k}{b}\right\rfloor b+k$, from \eqref{Cons2} we have
\begin{align*}
    [a_{i+h}-(i+h)]-[a_i-i]&=[a_{j+h}-(j+h)]-[a_j-j]+tmb,\\
    a_{i+h}-a_{i}-a_{j+h}+a_{j}&=tmb,
\end{align*}
implying that
\[
    a_{i+h}+a_{j}\equiv a_{j+h}+a_{i}\pmod {mb}.
\]
Because $\mA$ is a Sidon set in the additive group $\mathbb{Z}_{mb}$ we have $\{i+h,j\}=\{j+h,i\}$, that is $i=j$. Thus, $f$ is an $m\times n$ modular sonar sequence.
\end{proof}

\medskip
As a consequence of Theorem \ref{new sonar general}, and applying Theorems \ref{ConsBose} and \ref{ConsRuzsa}, and Propositions \ref{proposition Bose} and \ref{proposition Ruzsa}, we present the following new constructions of sonar sequences. Proofs follows immediately from Theorem \ref{new sonar general}.
\medskip

\begin{corollary}\label{Bose1}
Let $\mB(q,\theta,\alpha)$ be the set given in $\eqref{log B de alfa}$. 
The function $f:[1,q]\rightarrow\Z_{q-1}$ defined by $f(i):=\left\lfloor\frac{b_i}{q+1}\right\rfloor$, where $b_i\in \mB(q,\theta,\alpha)$ is the unique element such that $b_i\equiv i\bmod (q+1)$, is a $(q-1)\times q$ modular sonar sequence.
\end{corollary}
\medskip

\begin{corollary}\label{CRuzsa1}
Let $\mathcal{R}(p,\theta)$ be the set given in \eqref{conjunto de Ruzsa}. Then the function
\begin{enumerate}
    \item $f:[1,p-1]\rightarrow\Z_{p-1}$ defined by $f(i):=\left\lfloor\frac{r_i}{p}\right\rfloor$, where $r_i\in\mathcal{R}(p,\theta)$ is the unique element such that $r_i\equiv i\bmod p$, is a $(p-1)\times(p-1)$ modular sonar sequence.\medskip
    \item $f:[1,p-1]\rightarrow\Z_{p}$ defined by $f(i):=\left\lfloor\frac{r_i}{(p-1)}\right\rfloor$, where $r_i\in\mathcal{R}(p,\theta)$ is the unique element such that $r_i\equiv i\bmod(p-1)$, is a $p\times(p-1)$ modular sonar sequence.
\end{enumerate}
\end{corollary}
%
%
%
The following examples illustrates the Theorem \ref{new sonar general} and constructions given above.

\medskip

\begin{example}
Let $q=9$. From Bose's construction, $$\mathcal{B}=\{1,4,37,38,49,53,55,62,76\}$$ is a Sidon set with 9 elements in the additive group $\Z_{(10)(8)}$. From Corollary \ref{Bose1}, the function $f:[1,9]\rightarrow\Z_{8}$, defined by $f(i)=\lfloor b_i/8 \rfloor$, where $b_i\in\mathcal{B}$, defines a $8\times 9$ modular sonar sequence with 9 elements given by
\[
    [0,6,5,0,5,7,3,3,4].
\]
\end{example}
\begin{example}\

\begin{enumerate}
    \item Given $p=7$, from Ruzsa's construction we have that $$\mathcal{R}_1=\{2, 4, 5, 27, 31, 36\}$$ is a Sidon set with 6 elements in the additive group $\Z_{(7)(6)}$. From Corollary \ref{CRuzsa1} (item 1)), the function $f:[1,6]\rightarrow\Z_{6}$, defined by $f(i)=\lfloor r_i/7 \rfloor$, where $r_i\in\mathcal{R}_1$, defines a $6\times 6$ modular sonar sequence with 6 elements given by
        \[
            [5,0,4,0,0,3].
        \]
    \item Given $p=13$, from Ruzsa's construction we have that $$\mathcal{R}_2=\{10, 16, 57, 59, 90, 99, 115, 134, 144, 145, 149, 152\}$$ is a Sidon set with 12 elements in the additive group $\Z_{(13)(12)}$. Now, from Corollary \ref{CRuzsa1} (item 2)), the function $f:[1,12]\rightarrow\Z_{13}$, defined by $f(i)=\lfloor r_i/12 \rfloor$, where $r_i\in\mathcal{R}_2$, defines a $13\times 12$ modular sonar sequence with 12 elements given by
        \[
            [12,12,11,8,1,12,7,9,12,4,0,4]
        \]
\end{enumerate}
\end{example}

\section{Associated Problems}
The main problem in sonar sequences consists in to analyze the asymptotic behavior of the maximal functions
\begin{align*}
    G(m):=&\max \{n:\text{there exist a sonar sequence } m\times n\},\\
    G(\bmod{m}):=&\max \{n:\text{there existe a modular sonar sequence } m\times n\}.
\end{align*}
It is easy to prove that $G(m)\geq G\pmod{m}$. From the sonar sequences-constructions given in \cite{Sonar} and those presented in this paper, for any prime $p$ and any prime power $q=p^r$, we have the following relations
\begin{align*}
    G(\bmod{p})=p+1&\Rightarrow G(p)\geq p+1,\\
    G(\bmod(q-1))=q&\Rightarrow G(q-1)\geq q.
\end{align*}
So, is natural to state the following problem.\medskip\newline
\textbf{Problem 1.} Let $m$ be a positive integer. It is possible to prove or refute that $G(m)\geq m$. We conjecture that this is true.\medskip\newline
Now, we know that there exist modular sonar sequences for the following modulus: $p$ (Quadratic and Welch), $p-1$ (Welch and our Theorem \ref{new sonar general}), $q-1$ (Shift, Golomb, and our Theorem \ref{new sonar general}). Thus, we have the second problem.\medskip\newline
\textbf{Problem 2.} What can we say about sonar sequences--constructions for other modulus? Furthermore, is it possible to establish the exact value for $G(\bmod 35)$?\medskip\newline
We know several modular sonar sequences--constructions, but we don't know constructions for non--modular sonar sequences. So, we state the third problem.\medskip\newline
\textbf{Problem 3.} To establish constructions for $m\times n$ sonar sequences. It is possible to use Sidon sets, as we used in this paper, to find one of such constructions?\medskip\newline
Finnally, we are working in the following problem, with we hope to have new results.\medskip\newline
\textbf{Problem 4.} Make comparisons, using computer, among the known and new constructions of sonar sequences, to determine which of them produces better sequences.

\bigskip
\textbf{Acknowledgments.} Authors thank  to ``Patrimonio autónomo Fondo Nacional de Financiamiento para la Ciencia, la Tecnología y la Innovación--Francisco José de Caldas'' and COLCIENCIAS for financial support under research project 110356935047.
\bibliographystyle{ieeetr}

\end{document}